\documentclass[11pt,letterpaper,reqno]{amsart}
\usepackage{fullpage}
\usepackage{amsmath,amsthm,amsfonts,amssymb,amscd}
\usepackage{empheq}
\usepackage{thmtools}
\usepackage{hyperref}
\usepackage{makecell}
\usepackage{enumerate}
\usepackage{enumitem}
\usepackage{bbm}
\usepackage{array}
\usepackage{float}

\allowdisplaybreaks
\renewcommand{\arraystretch}{1}

\let\subsectiontemp\subsection
\renewcommand{\subsection}[1]{ 
    \subsectiontemp{#1} \hfill\vspace{0.5\linespacing} 
} 

\hypersetup{
  colorlinks=true,
  linkcolor=blue,
  citecolor=blue,
  linkbordercolor={0 0 1}
}

\newtheorem{theorem}{Theorem}[section]

\newtheorem{lemma}[theorem]{Lemma}

\mathtoolsset{showonlyrefs=true}
\numberwithin{equation}{section}
\numberwithin{figure}{section}
\numberwithin{table}{section}

\newcommand{\lrp}[1]{\left(#1\right)}
\newcommand{\lrb}[1]{\left[#1\right]}
\newcommand{\lrcb}[1]{\left\{#1\right\}}

\newcommand{\pttmatrix}[4]{
\left(
\begin{smallmatrix}
  #1 & #2 \\
  #3 & #4
\end{smallmatrix}
\right)
}

\newcommand{\ZZ}{\mathbb{Z}}
\newcommand{\CC}{\mathbb{C}}
\newcommand{\QQ}{\mathbb{Q}}

\newcommand{\SL}{\text{SL}}

\newcommand{\lcm}{\textnormal{lcm}}

\newcommand{\FF}{\mathcal{F}}
\newcommand{\HH}{\mathcal{H}}
\newcommand{\CM}{\textnormal{CM}}
\newcommand{\Or}[1]{\mathcal{O}_{#1}}
\newcommand{\Cl}[1]{\textnormal{Cl}(\mathcal{O}_{#1})}
\newcommand{\PrinCl}[1]{\text{PrinCl}(\mathcal{O}_{#1})}
\newcommand{\PreCl}[1]{\text{PreCl}(\mathcal{O}_{#1})}
\newcommand{\Exp}{\textnormal{Exp}\,}

\def \t {\tau}

\author{David Aiken}
\address[D. Aiken]{Department of Mathematics, Louisiana State University,
Baton Rouge, Louisiana}

\author[E. Ross]{Erick Ross}
\address[E. Ross]{School of Mathematical and Statistical Sciences, Clemson University, Clemson, SC}
\email{erickjohnross@gmail.com}

\author[D. Shvydkoy]{Dmitriy Shvydkoy}
\address[D. Shvydkoy]{Department of Mathematics, University of Illinois Urbana-Champaign,
Urbana, IL}

\author[H. Xue]{Hui Xue}
\address[H. Xue]{School of Mathematical and Statistical Sciences, Clemson University, Clemson, SC}
\email{huixue@clemson.edu}



\title{Boundary CM points and class groups of small exponent}

\begin{document}

\begin{abstract}
Let $\mathcal F$ denote the fundamental domain for $\text{SL}_2(\mathbb{Z})$ on the upper half plane $\mathcal H$.
William Duke showed that as fundamental discriminants $D \to -\infty$, the sets $\textnormal{CM}_{D}$ (CM points of discriminant $D$) are equidistributed in $\mathcal F$. In this paper, we investigate the behavior of CM points on the boundary of $\mathcal F$.
We prove that such CM points are equidistributed on the boundary, and also give a complete characterization of when every $\textnormal{CM}_D$ point lies on the boundary. 
Along the way, we also (conditionally) give a complete classification of 
negative discriminants with class group of small exponent.
\end{abstract}

\maketitle

\section{Introduction}

Let $\HH \coloneq \{z \in \CC : \text{Im } z > 0 \}$ denote the complex upper half plane, and 
\[
    \FF \coloneq  \lrcb{z = x+iy \in \mathcal{H} : |z| \geq 1, -\frac{1}{2} \leq x \leq 0} \cup \lrcb{z=x+iy \in \mathcal{H} : |z| > 1, 0 < x < \frac{1}{2}}
\]
denote the standard fundamental domain for $\SL_2(\ZZ)$ acting on $\HH$. 
For negative discriminants $D$, let $\CM_D$ denote the set of CM points in $\HH$ of discriminant $D$. It turns out that the action of $\SL_2(\ZZ)$ on $\HH$ is also an action on these $\CM_D$ points.  
Hereafter, when we refer to ``discriminants" in this paper, we are only referring to negative discriminants.

In \cite{duke}, Duke investigated how $\CM$ points are distributed over the fundamental domain $\FF$. In particular, he showed that for fundamental discriminants $D$; as $D \to -\infty$, the $\CM_D$ points are equidistributed via the hyperbolic measure $d\mu_{\textnormal{hyp}} = \frac{1}{y^2} \,dx\,dy$.
In this paper, we investigate the behavior of CM points on the boundary of $\FF$. We study this question for all discriminants $D$, not just the fundamental ones.


First, we show that the CM points on the boundary of $\FF$ are equidistributed according to certain metrics. We will parametrize CM points on the left boundary of $\FF$ (i.e. $\{-\frac{1}{2} + iy \ :\ y \ge \frac{\sqrt{3}}{2}\}$) by their imaginary parts $y$, and CM points on the lower arc of $\FF$ (i.e. $\{e^{i\theta} \ :\ \frac{\pi}{2} \le \theta \le \frac{2\pi}{3} \}$) by their arguments $\theta$.
\begin{theorem} \label{thm:equidistribution-boundary}
    Let $\CM^\textnormal{LB}_{|D| \le \Delta}$ denote the CM points on the left boundary of $\FF$ with absolute discriminant $|D| \le \Delta$, and $\CM^\textnormal{LA}_{|D| \le \Delta}$ denote the CM points on the lower arc of $\FF$ with absolute discriminant $|D| \le \Delta$. Then as $\Delta \to \infty$;
    \begin{enumerate}[leftmargin=*]
        \item $\textnormal{Im\,} \CM^\textnormal{LB}_{|D|\le \Delta}$ is equidistributed over every fixed $[\frac{\sqrt 3}{2}, K]$ via the metric $d\mu = \frac{1}{y}dy$,
        \item $\textnormal{arg\,} \CM^\textnormal{LA}_{|D|\le \Delta}$ is equidistributed over $[\frac{\pi}{2}, \frac{2\pi}{3}]$ via the metric $d\nu = \frac{1}{\sin \theta} \,d\theta$.
    \end{enumerate}
\end{theorem}
We make two remarks here about this result.
\begin{enumerate}
    \item 
    This equidistribution theorem differs  from Duke's result in that we are considering $\CM_{D}$ points for all discriminants $|D| \le \Delta$, instead of for each individual discriminant $D$. This is necessary because for individual discriminants $D$, the $\CM_D$ points do not even become dense along the boundary of $\mathcal F$ (for example, there is only one $\CM_D$ point on the boundary of $\FF$ for negative prime discriminants $D$).
    \item 
    The two parts of Theorem \ref{thm:equidistribution-boundary} can actually be viewed as the same phenomenon.
    The two metrics $d\mu$ and $d\nu$ are precisely the restriction of the hyperbolic metric $ds=\frac{\sqrt{dx^2+dy^2}}{y}$ to the left boundary of $\FF$ and to the lower arc of $\FF$, respectively. And in fact, we will prove both parts of Theorem \ref{thm:equidistribution-boundary} in a unified manner (by transforming the lower arc of $\FF$ into the vertical line segment $\{-\frac{1}{2}+iy : \frac{1}{2} \le y \le \frac{\sqrt 3}{2}\}$). 
\end{enumerate}

Next, in Theorem \ref{thm:all-CMpts-on-bdry}, we give a classification of when all the $\CM_D$ points in $\FF$ lie on the boundary of $\FF$. This theorem utilizes the fact that the $\CM_D$ points lying on the boundary of $\FF$ are precisely those of order dividing $2$ (by Lemma \ref{lem:location-of-order-2-CM-pts}). 

\begin{theorem} \label{thm:all-CMpts-on-bdry}
    Every $\CM_D$ point contained in $\FF$ lies on the boundary of $\FF$ if and only if $\Cl{D}$ has exponent dividing $2$ and $D$ is odd or $-4$.

    Conditionally on the non-existence of Siegel zeros, such discriminants $D$ are precisely those given in the following table. Unconditionally, this table could possibly also include the discriminants arising from one additional fundamental discriminant.

    \begin{center}
    \def\arraystretch{1.4}
    \begin{tabular}{ c | c } 
         Class Group & Discriminants \\ 
         \hline
         $\{\textnormal{e}\}$ &  $-3, -4, -7, -11, -19, -27, -43, -67, -163$\\  
         $\ZZ/2\ZZ$ &  $-15, -35, -51, -75, -91, -99, -115, -123, -147, -187, -235, -267, -403, -427$\\  
         $(\ZZ/2\ZZ)^2$ & $-195, -315, -435, -483, -555, -595, -627, -715, -795, -1435$ \\  
         $(\ZZ/2\ZZ)^3$ & $-1155, -1995, -3003, -3315$ \\
    \end{tabular}
    \end{center}
    \vspace{3mm}
\end{theorem}

In order to show the second half of this theorem, we will need a classification of class groups $\Cl{D}$ of exponent dividing $2$. (Recall here that the exponent $E$ of a group $G$ is the minimal integer $E \ge 1$ such that $g^E = e$ for all $g \in G$.)
The methods we develop to give this classification turn out to also work for any given exponent $E$. So we give a (conditional) classification of all class groups $\Cl{D}$ with exponent $1 \le E \le 8$.  

Such a classification was already given for fundamental discriminants in \cite{fieldexponents}.  
To extend from fundamental discriminants to all discriminants, we use the ideas from a recent work of Fan and Pollack \cite{FanPollack}. In particular, for any given fundamental discriminant $D_0$, we use ideas from \cite{FanPollack} to develop an algorithm to compute the full list of conductors $f$ for $D_0$ such that $\Cl{D_0 f^2}$ has exponent $E$.

Computing discriminants via this algorithm then yields Tables \ref{table:CMexp1} - \ref{table:CMexp8}, listing discriminants with class group of exponent $1 \le E \le 8$. 
The following theorem gives conditions for which one can prove that these tables are complete. (And in the authors' opinion, it is extremely likely that all of these tables are in fact complete.)
\begin{theorem} \label{thm:exp-tables-complete-under-certain-assumptions}
    Tables \ref{table:CMexp1} - \ref{table:CMexp8} list negative discriminants with class group of exponent $1 \le E \le 8$. 
    \begin{enumerate}
        \item Unconditionally, the table for exponent $E=1$ is complete. 
        \item Assuming the non-existence of Siegel zeros, the tables for exponent $E=2,4,8$ are all complete. 
        \item Unconditionally, the tables for exponent $E=2,4,8$ could also contain the discriminants arising from at most one additional fundamental discriminant.
        \item Assuming ERH (Extended Riemann Hypothesis), the tables for exponent $E = 2,3,4,5,8$ are all complete. 
        \item Assuming ERH, the table for $E=6$ is complete if there are no fundamental discriminants $ 3.1\cdot 10^{20} \le |D| \le 2.5 \cdot 10^{25}$ with class groups of exponent $6$.
        \item Assuming ERH, the table for $E=7$ is complete if there are no fundamental discriminants $ 3.1\cdot 10^{20} \le |D| \le 3.9 \cdot 10^{30}$ with class groups of exponents $7$.
    \end{enumerate}
\end{theorem}

Now, we would like to point out three different motivations for this project investigating CM points on the boundary of $\FF$. 
First, this project is a natural variation of Duke's celebrated result in \cite{duke} concerning the equidistribution of CM points over $\FF$. Theorem \ref{thm:equidistribution-boundary} here gives an analogous result about the equidistribution of CM points over the boundary of $\FF$. 
Second, it leads to a classification of all discriminants of small exponent, which is of independent interest (compare, for example, with \cite{fieldexponents}).
Third, the results of this paper were needed in \cite{our-first-paper} to show certain results about transcendence of zeros of modular forms. 
(In fact, this was the original reason why we started to investigate this topic of CM points on the boundary of $\FF$.) 
It turns out that in certain cases, if a modular form has a non-transcendental zero, then it has zeros at every $\CM_D$ point for some $D$ \cite[Lemma 3.1]{our-first-paper}. But in many scenarios, these modular forms can only have zeros on the boundary of $\FF$. 
For example, this property was shown for Eisenstein series $E_k$ in \cite{rankin-swinnerton-dyer} and for cuspidal projections of products of Eisenstein series $\Delta_{k,\ell} = E_k E_\ell - E_{k+\ell}$ in \cite{reitzes-vulakh-young,xue-zhu}. 
Hence it becomes necessary to classify which $\CM_D$ have all their points on the boundary of $\FF$.

Finally, we give an overview of the structure of the paper. In Section \ref{sec:CMthry}, we give some background, reviewing the basics of CM theory. Then in Section \ref{sec:equidistribution-boundary}, we show Theorem \ref{thm:equidistribution-boundary}  by estimating how many CM points lie in given intervals. In Section \ref{sec:proof-of-thm1.1-(all-CMpts-on-bdry)}, we prove Theorem \ref{thm:all-CMpts-on-bdry} (assuming Theorem \ref{thm:exp-tables-complete-under-certain-assumptions}). Finally in Sections \ref{sec:property-of-conductors} and \ref{sec:algorithm-to-compute-discriminants} we prove Theorem \ref{thm:exp-tables-complete-under-certain-assumptions}. Section \ref{sec:property-of-conductors} uses ideas from \cite{FanPollack} to prove a particular divisibility property (Lemma \ref{lem:property-of-f}) that conductors $f$ must satisfy in order to yield class groups of a given exponent. Then Section \ref{sec:algorithm-to-compute-discriminants} uses this divisibility property to develop an algorithm to compute the complete list of conductors with class group of a given exponent. This algorithm is then used to compute Tables \ref{table:CMexp1} - \ref{table:CMexp8}, listed in the appendix.


\section{Background on CM Theory} \label{sec:CMthry}

In this section, we review the basics of CM theory. The details are all standard, and can be found in Cox \cite{cox_primes_1989}.

Recall that a CM point $\t$ is the solution in $\HH$ of a quadratic equation with integer coefficients $az^2 + bz + c = 0$. Without loss of generality, we impose here the conditions $a\ge1$ and $\gcd(a,b,c) = 1$.
This CM point $\tau$ is uniquely determined by the triple $[a,b,c]$,
so we will often refer to a CM point $\tau$ by its corresponding triple $[a,b,c]$. Here, $D \coloneq b^2-4ac$ is called the \textit{discriminant} of $\tau = [a,b,c]$, and we let $\CM_D$ denote the set of CM points with discriminant $D$. Note that the discriminant $D$ here is necessarily negative and congruent to $0$ or $1$ modulo $4$.

We call two CM points equivalent if they are in the same orbit under the standard action of $\SL_2(\ZZ)$ on $\HH$: $\pttmatrix{p}{q}{r}{s}: \t \mapsto \frac{p\t+q}{r\t+s}$.
It turns out that mapping a CM point $\tau$ to the ideal $\ZZ \oplus \tau \ZZ$ induces a bijection between $\SL_2(\ZZ)$-equivalence classes of $\CM_D$ and the class group $\Cl{D}$. Under this bijection, $\SL_2(\ZZ) \backslash \CM_D$ inherits the group structure of $\Cl{D}$. 
Here, $\Cl{D}$ denotes the class group of the imaginary quadratic order $\mathcal O_D$ of discriminant $D$.

Now, a discriminant $D_0$ is called \textit{fundamental} if it is the discriminant of the quadratic field $\QQ(\sqrt{D_0})$. In this case, $\Or{D_0}$ would be the full ring of integers of $\QQ(\sqrt{D_0})$, and so $\Cl{D_0} =\text{Cl}(\QQ(\sqrt{D_0}))$.
It turns out that every discriminant $D$ can be factored uniquely as $D = D_0f^2$, where $D_0$ is a fundamental discriminant, and $f$ is called the conductor. Under this factorization, for any $\t$ such that $\Or{D_0} = \ZZ \oplus \tau \ZZ$, we have $\Or{D} = \ZZ \oplus f \tau \ZZ$.

Finally, we mention two facts about the group structure of $\SL_2(\ZZ) \backslash \CM_D$. First, the group inverse of the $\CM_D$ point $[a,b,c]$ is $[a,-b,c]$. Second, the $\CM_D$ group identity in $\FF$ is given by
\begin{equation} \label{eqn:identity-CM-pts}
    \tau = \begin{cases}
        [1, 0, -\frac{D}{4}] & \text{if} \quad D \equiv 0 \pmod{4} \\
        [1, 1, -\frac{D-1}{4}] & \text{if} \quad  D \equiv 1 \pmod{4}. \\
    \end{cases}
\end{equation}
Observe that when $D \equiv 0 \pmod{4}$, this point lies on the imaginary axis $\{iy \ :\ y \ge 1\}$, and when $D \equiv 1 \pmod{4}$, it lies on the left boundary $\{-\frac{1}{2} + iy \ :\ y \ge \frac{\sqrt{3}}{2}\}$ of $\FF$.


\section{Proof of Theorem \ref{thm:equidistribution-boundary}} \label{sec:equidistribution-boundary}

In this section, we prove Theorem \ref{thm:equidistribution-boundary}.
We will utilize the following three lemmas, which all follow from standard analytic number theory arguments.
In the following, $a,c,d$ will always denote natural numbers. Additionally, $\phi$ denotes the Euler totient function, $\mu$ denotes the Mobius function, and $\gamma$ denotes the Euler-Mascheroni constant.
\begin{lemma}[{\cite[p. 393]{niven-zuckerman-montgomery}}] \label{lem:count-coprime-range}
    For positive real numbers $T$ and natural numbers $a$,
    \begin{align}
        \#\{c \le T ~:~ (c,a)=1\} = \frac{\phi(a)}{a} T + O(a^{1/4}).
    \end{align}
\end{lemma}

\begin{lemma}[{\cite[p. 399]{niven-zuckerman-montgomery}}] \label{lem:sum-phi}
    For positive real numbers $T$,
    \begin{align}
        \sum_{a \le T} \phi(a) &= \frac{3}{\pi^2} T^2 + O(T \log T).
    \end{align}
\end{lemma}

\begin{lemma} \label{lem:sum-phi-over-a2}
    For positive real numbers $T$,
    \begin{align}
        \sum_{a \le T} \frac{\phi(a)}{a^2} &= \frac{6}{\pi^2} \log T + \gamma_0 + O\lrp{\frac{\log T}{T}} 
    \end{align}
    for a certain constant $\gamma_0$.
\end{lemma}
\begin{proof}
    We have 
    \begin{align}
        \sum_{a \le T} \frac{\phi(a)}{a^2} &= \sum_{a \le T} \frac{1}{a} \sum_{d|a} \frac{\mu(d)}{d} = \sum_{d \le T} \frac{\mu(d)}{d} \sum_{\substack{a \le T \\ d|a}} \frac{1}{a} = \sum_{d \le T} \frac{\mu(d)}{d^2} \sum_{a' \le T/d } \frac{1}{a'}  \\
        &= \sum_{d \le T} \frac{\mu(d)}{d^2} \lrp{\log (T/d) + \gamma + O\lrp{\frac{1}{T/d}} }
        \\
        &= (\log T) \sum_{d \le T} \frac{\mu(d)}{d^2} + \sum_{d \le T} \frac{\mu(d)(\gamma-\log d)}{d^2} + \sum_{d \le T} \frac{\mu(d)}{d} O\lrp{\frac1T}
        \\
        &= (\log T)\lrp{\frac{6}{\pi^2} + O\lrp{\frac1T}} + \lrp{\gamma_0 + O\lrp{\frac{\log T}{T}}} + O\lrp{\frac{\log T}{T}} \\
        &= \frac{6}{\pi^2} \log T + \gamma_0 + O\lrp{\frac{\log T}{T}},
    \end{align}
    as desired.
    Note that $\gamma_0$ here is the constant
    $\displaystyle
        \gamma_0 := \sum_{d \ge 1} \frac{\mu(d)(\gamma-\log d)}{d^2}.
    $
\end{proof}

We now use these three lemmas to show that CM points are equidistributed over the boundary of $\FF$.

{
\renewcommand{\thetheorem}{\ref{thm:equidistribution-boundary}}
\begin{theorem} 
    Let $\CM^\textnormal{LB}_{|D| \le \Delta}$ denote the CM points on the left boundary of $\FF$ with absolute discriminant $|D| \le \Delta$, and $\CM^\textnormal{LA}_{|D| \le \Delta}$ denote the CM points on the lower arc of $\FF$ with absolute discriminant $|D| \le \Delta$. Then as $\Delta \to \infty$;
    \begin{enumerate}[leftmargin=*]
        \item $\textnormal{Im\,} \CM^\textnormal{LB}_{|D|\le \Delta}$ is equidistributed over every fixed $[\frac{\sqrt 3}{2}, K]$ via the metric $d\mu = \frac{1}{y}dy$,
        \item $\textnormal{arg\,} \CM^\textnormal{LA}_{|D|\le \Delta}$ is equidistributed over $[\frac{\pi}{2}, \frac{2\pi}{3}]$ via the metric $d\nu = \frac{1}{\sin \theta} \,d\theta$.
    \end{enumerate}
\end{theorem}
\addtocounter{theorem}{-1}
}
\begin{proof}
    First, observe that under the $\SL_2(\ZZ)$ transformation $z \mapsto \pttmatrix{0}{-1}{1}{1}z = \frac{-1}{z+1}$, the lower arc $\{e^{i\theta} \ :\ \frac{\pi}{2} \le \theta \le \frac{2\pi}{3} \}$ maps bijectively to $\{-\frac{1}{2} + iy \ :\ \frac12  \le y \le \frac{\sqrt{3}}{2}\}$. Additionally this transformation maps $\theta \mapsto y = \frac 12 \frac{\sin \theta}{1+\cos \theta}$, so that the metric $d\nu = \frac{1}{\sin \theta} d\theta$ becomes $d\mu = \frac{1}{y}dy$. So let $\CM^\textnormal{LB*}_{|D| \le \Delta}$ denote the CM points on $\{-\frac{1}{2} + i y : y \ge \frac12\}$ with absolute discriminant $|D| \le \Delta$.
    Then to prove parts (1) and (2), it suffices to show that
    $\textnormal{Im\,} \CM^\textnormal{LB*}_{|D|\le \Delta}$ is equidistributed over every fixed $[\frac{1}{2}, K]$ via the metric $d\mu = \frac{1}{y}dy$.

    Now, it is straightforward to see that
    \begin{align}
        \textnormal{Im\,}\CM^\textnormal{LB*}_{|D|\le \Delta} 
        &= \lrcb{ \textnormal{Im\,} [a,a,c] ~:~ a \ge 1,~ c \ge \frac{a}{2},~ (c,a)=1,~ |D| \le \Delta } \\
        &= \lrcb{ \frac{\sqrt{4ac-a^2}}{2a} ~:~ a \ge 1,~ c \ge \frac{a}{2},~ (c,a)=1,~ 4ac-a^2 \le \Delta }. 
    \end{align}
    Then under the reparametrization $t(y) = y^2+\frac14$, the set $\textnormal{Im\,}\CM^\textnormal{LB*}_{|D|\le \Delta} \cap [\frac{1}{2}, K]$ maps precisely to the set of rationals:
    \begin{align}
        R_{\Delta} := \lrcb{\frac12 \le \frac{c}{a} \le K' ~:~ a\ge 1,~(c,a)=1, ~4ac-a^2 \le \Delta} \qquad\text{(where $K':=K^2+1/4$)}.
    \end{align} 
    Also note that under this reparametrization, $d\mu = \frac{1}{y}\,dy$ becomes $d\mu = \frac{2}{4t-1} \,dt$.
    
    Hence to prove the desired result, it suffices to show that $R_{\Delta}$ is equidistributed over $[\frac12,K']$ via the metric $d\mu = \frac{2}{4t-1} \,dt$.  
    Specifically, this means that for every fixed interval $[X,Y] \subseteq [1,K']$,
    \begin{align}
        \frac{\#\, R_{\Delta} \cap [X,Y]}{\#\,R_{\Delta}} \longrightarrow \frac{\int_X^Y d\mu}{\int_1^{K'} d\mu} = 
        \frac{
            \lrb{\frac{1}{2} \log(4t-1)}_{t=X}^{t=Y} 
        }{
            \lrb{\frac{1}{2} \log(4t-1)}_{t=1}^{t=K'}
        } \qquad \text{as} \qquad \Delta\to\infty.
    \end{align} 
    We will prove this identity by showing that
    \begin{align} \label{eqn:goal-temp1} \tag{$*$}
        \#\, R_{\Delta} \cap [X,Y] = \frac{3\Delta}{2\pi^2} \lrb{\frac12\log(4t-1)}_{t=X}^{t=Y} + O\lrp{\Delta^{5/8}}.
    \end{align}
    
    Now, observe that $4ac-a^2 \le \Delta$ if and only if $c \le \frac{\Delta}{4a}+\frac{a}{4}$. This means that
    \begin{align} 
        \#\, R_{\Delta} \cap [X,Y] 
        &= \sum_{a \ge 1} \#\lrcb{ X \le \frac{c}{a} \le Y ~:~ (c,a)=1,\ 4ac-a^2 \le \Delta} \\
        &= \sum_{a \ge 1} \#\lrcb{ aX \le c \le \min\lrp{aY, \frac{\Delta}{4a}+\frac{a}{4}} ~:~ (c,a)=1}.
        \label{eqn:temp1.1}
    \end{align}
    Next, observe that 
    \begin{align}
        aY \le \frac{\Delta}{4a} + \frac{a}{4} \quad\text{iff}\quad a \le \sqrt{\frac{\Delta}{4Y-1}} \qquad
        \text{and}\qquad
        aX \le \frac{\Delta}{4a} + \frac{a}{4} \quad\text{iff}\quad a \le \sqrt{\frac{\Delta}{4X-1}} .
    \end{align}
    Applying these facts to \eqref{eqn:temp1.1}, we obtain that
    \begin{align} 
        &\#\, R_{\Delta} \cap [X,Y]  \\
        =&\qquad 
        \sum_{1 \le a \le \sqrt{\frac{\Delta}{4Y-1}}} \#\lrcb{ aX \le c \le aY  ~:~ (c,  a)=1}  \\
        &+
        \sum_{\sqrt{\frac{\Delta}{4Y-1}} < a \le \sqrt{\frac{\Delta}{4X-1}}} \#\lrcb{ aX \le c \le  \frac{\Delta}{4a}+\frac{a}{4} ~:~ (c,a)=1} \\
        =&\qquad \sum_{1 \le a \le \sqrt{\frac{\Delta}{4Y-1}}} \lrb{\frac{\phi(a)}{a} \lrp{aY-aX} + O\lrp{a^{1/4}}} 
        \qquad\qquad\qquad \text{(Lemma \ref{lem:count-coprime-range})}\\
        &+
        \sum_{\sqrt{\frac{\Delta}{4Y-1}} < a \le \sqrt{\frac{\Delta}{4X-1}}}  \lrb{\frac{\phi(a)}{a} \lrp{\frac{\Delta}{4a}+\frac{a}{4} - aX} + O\lrp{a^{1/4}} }
        \qquad\text{(Lemma \ref{lem:count-coprime-range})}\\
        =&\qquad (Y-X) \sum_{1 \le a \le \sqrt{\frac{\Delta}{4Y-1}}} \phi(a) 
        \quad+\quad \lrp{\frac14-X} \sum_{\sqrt{\frac{\Delta}{4Y-1}} < a \le \sqrt{\frac{\Delta}{4X-1}}} \phi(a)
        \\
        &+\quad
        \frac{\Delta}{4} \sum_{\sqrt{\frac{\Delta}{4Y-1}} < a \le \sqrt{\frac{\Delta}{4X-1}}}  \frac{\phi(a)}{a^2}  \quad+\quad O\lrp{\sqrt{\Delta} \cdot \Delta^{1/8}} \\
        =&\, (Y-X) \frac{3}{\pi^2} \lrb{\frac{\Delta}{4Y-1}} 
        + \lrp{\frac14-X} \frac{3}{\pi^2}\lrb{\frac{\Delta}{4X-1} - \frac{\Delta}{4Y-1}}
        \quad\text{(Lemma \ref{lem:sum-phi})}
        \\
        &+
        \frac{\Delta}{4} \frac{6}{\pi^2}\lrb{ 
        \log \sqrt{\frac{\Delta}{4X-1}}
        - \log \sqrt{\frac{\Delta}{4Y-1}}
        } + O\lrp{\Delta^{5/8}} \quad \qquad\ \ \text{(Lemma \ref{lem:sum-phi-over-a2})} \\
        =&\, \frac{3\Delta}{\pi^2} \lrb{
            \frac{Y-X}{4Y-1} - \frac{1}{4} +  \frac{X-\frac14}{4Y-1}
        } 
        \\
        &+
        \frac{\Delta}{4} \frac{6}{\pi^2} \lrb{ 
        \frac12 \log (4Y-1) - \frac12 \log (4X-1)
        } + O\lrp{\Delta^{5/8}} \\
        =&\, \frac{3\Delta}{2\pi^2} \lrb{\frac12 \log(4t-1)}_{t=X}^{t=Y} + O\lrp{\Delta^{5/8}},
    \end{align}
    verifying \eqref{eqn:goal-temp1}.
\end{proof}


\section{Proof of Theorem \ref{thm:all-CMpts-on-bdry}} \label{sec:proof-of-thm1.1-(all-CMpts-on-bdry)}

We first give a lemma characterizing the CM points that lie on the boundary of $\FF$ (or on the imaginary axis).

\begin{lemma} \label{lem:location-of-order-2-CM-pts}
    The $\CM_D$ points in $\FF$ lying on the boundary of $\FF$ or on the imaginary axis are precisely those of order dividing $2$. 
    Moreover, $\CM_D$ includes a point on the imaginary axis only when $D \equiv 0 \pmod{4}$.
\end{lemma}

\begin{proof}
Recall that the inverse of $[a, b, c]$ is $[a, -b, c]$, and observe that this operation corresponds to reflecting points in $\HH$ across the imaginary axis. This means that a $\CM_D$ point of order dividing $2$ must necessarily be  $\SL_2(\ZZ)$-equivalent to its reflection. Hence such $\CM_D$ points in $\FF$ are precisely those that lie on the boundary of $\FF$ or on the imaginary axis. 

Moreover, note that every $\CM_D$ point on the imaginary axis is of the form $[a,0,c]$. In this case, the discriminant $D = -4ac$ must be congruent to $0$ modulo $4$.
\end{proof}

This lemma allows us to classify all discriminants $D$ satisfying the property that every $\CM_D$ point in $\FF$ lies on the boundary of $\FF$. In particular, we prove Theorem \ref{thm:all-CMpts-on-bdry} here, assuming Theorem \ref{thm:exp-tables-complete-under-certain-assumptions}. Then Theorem \ref{thm:exp-tables-complete-under-certain-assumptions} will be proven in the last two sections of the paper. 

{
\renewcommand{\thetheorem}{\ref{thm:all-CMpts-on-bdry}}
\begin{theorem}
    Consider negative discriminants $D$. Then every $\CM_D$ point in $\FF$ lies on the boundary of $\FF$ if and only if $\Cl{D}$ has exponent dividing $2$ and $D$ is odd or $-4$.

    Conditionally on the non-existence of Siegel zeros, such discriminants $D$ are precisely those given in the following table. Unconditionally, this table could possibly also include the discriminants arising from one additional fundamental discriminant.

    {
    \centering
    \def\arraystretch{1.4}
    \begin{tabular}{ c | c } 
         Class Group & Discriminants \\ 
         \hline
         $\textnormal{\{e\}}$ &  $-3, -4, -7, -11, -19, -27, -43, -67, -163$\\  
         $\ZZ/2\ZZ$ &  $-15, -35, -51, -75, -91, -99, -115, -123, -147, -187, -235, -267, -403, -427$\\  
         $(\ZZ/2\ZZ)^2$ & $-195, -315, -435, -483, -555, -595, -627, -715, -795, -1435$ \\  
         $(\ZZ/2\ZZ)^3$ & $-1155, -1995, -3003, -3315$ \\
    \end{tabular}
    }
\end{theorem}
\addtocounter{theorem}{-1}
}

\begin{proof}   
    The first half of the theorem for even discriminants follows from \eqref{eqn:identity-CM-pts}. In this case, the $\CM_D$ identity $[1, 0, -\frac{D}{4}] = \frac{\sqrt{D}}{2}$ lies on the boundary of $\FF$ only for $D=-4$. And since $|\Cl{-4}| = 1$, $\CM_{-4}$ consists of just this point.

    The first half of the theorem for odd discriminants follows from Lemma \ref{lem:location-of-order-2-CM-pts}. In this case, a $\CM_D$ point in $\FF$ lies on the boundary of $\FF$ if and only if it is of order dividing $2$. This immediately yields the desired result. 

    Finally, the second half of the theorem follows from Theorem \ref{thm:exp-tables-complete-under-certain-assumptions}. Here, the desired table of discriminants comes from the classification of class groups of exponent $1$ and $2$.
\end{proof}


\section{A Divisibility Property for Conductors} \label{sec:property-of-conductors}

In these last two sections, we work towards a proof of Theorem \ref{thm:exp-tables-complete-under-certain-assumptions}. This section uses ideas from Fan and Pollack \cite{FanPollack} to show that conductors with class groups of a given exponent must satisfy a certain divisibility property (Lemma \ref{lem:property-of-f}). 
It turns out that for any given exponent $E$, there are only finitely many conductors $f$ that satisfy this divisibility property.
Hence in the next section, we develop an algorithm to compute the complete list of conductors $f$ associated to a fundamental discriminant $D_0$ such that $\Cl{D_0 f^2}$ has exponent $E$. This algorithm will then be used to construct Tables \ref{table:CMexp1} - \ref{table:CMexp8}.

We now work towards a proof of Lemma \ref{lem:property-of-f}.
Fix a fundamental discriminant $D_0$, and let $D = D_0f^2$ be a discriminant with conductor $f$. The following exposition follows Fan and Pollack \cite{FanPollack}.
Define the \textit{principal part} of the class group as follows:
\[
    \PrinCl{D} \coloneq (\Or{D_0} / f\Or{D_0})^\times / \langle 
    \text{images of units of } \Or{D_0} \text{ and images of integers coprime to } f  \rangle.
\]
In particular, $\PrinCl{D}$ can be identified with a subgroup of $\Cl{D}$. Next, we define the \textit{pre-class group} by
\[
     \PreCl{D} \coloneq (\Or{D_0} / f\Or{D_0})^\times / \langle \text{images of integers coprime to } f \rangle.
\]
It follows from the definitions that
\begin{equation} \label{eqn:pre-prin-correspondance}
    \PreCl{D} / U_f \cong \PrinCl{D}
\end{equation}
where $U_f$ is the image of $\Or{D_0}^\times$ in $\PreCl{D}$. So in particular, $|U_f|$ divides $|\Or{D_0}^\times|$. Note that for imaginary quadratic fields, we have the following formula for $|\Or{D_0}^\times|$:
\[
    |\Or{D_0}^\times| = \begin{cases}
        6 & D_0 = -3 \\
        4 & D_0 = -4 \\
        2 & \text{otherwise}.
    \end{cases}
\]
Now, define the function
\[
    L(f) \coloneq \lcm\lrcb{ p^k\lrp{1 - \lrp{\frac{D_0}{p}} \frac{1}{p}} \ \  \colon\ \  p^k || f}.
\]
We now show the following lemma that bounds $L(f)$ in terms of $\Exp\Cl{D}$. 
This is the divisibility property referenced previously that all conductors $f$ with class group of a given exponent must satisfy.
\begin{lemma} \label{lem:property-of-f}
    Fix a negative discriminant $D = D_0f^2$. Then $L(f)$ divides $12\cdot |\Or{D_0}^\times| \cdot \Exp\Cl{D}$.
\end{lemma}

\begin{proof}
    First, by \cite[Proposition 2.3]{FanPollack}, we have
    \[
        L(f) \ \bigg|\  12\cdot \Exp \PreCl{D}.
    \]
    Next, as a consequence of \eqref{eqn:pre-prin-correspondance}, we have
    \[
        \Exp\PreCl{D} \ \bigg|\  |U_f| \cdot \Exp\PrinCl{D} \ \bigg| \ |\Or{D_0}^\times| \cdot  \Exp\PrinCl{D}.
    \]
    Finally, since $\PrinCl{D}$ is a subgroup of $\Cl{D}$,
    \[
        \Exp\PrinCl{D} \ \bigg|\  \Exp\Cl{D}.
    \]
    Combining these three statements,
    \[
        L(f) \ \bigg|\  12\cdot |\Or{D_0}^\times| \cdot \Exp\Cl{D},
    \]
    as desired.
\end{proof}

Finally, for any fixed exponent $E$, we take note of a method to determine which $f$ satisfy the divisibility property $L(f) \mid  (12\cdot |\Or{D_0}^\times| \cdot E)$ from Lemma \ref{lem:property-of-f}. This problem reduces to the question of determining which $f$ can take on a given value of $L(f)$. And this question is answered by the following lemma, essentially amounting to a way to invert the function $f \mapsto L(f)$. 

\begin{lemma} \label{lem:f-divides-L(theta(f))}
    Let $\displaystyle
    \theta(m) \coloneq \prod_{L(p^k)| m} p^k$. Then
    $f$ divides $\theta(L(f))$.
\end{lemma}
This lemma follows immediately from the fact that $L(p^k) \mid  L(f)$ for each $p^k \parallel f$.

\section{An algorithm to compute discriminants} \label{sec:algorithm-to-compute-discriminants}

Fix $1 \le E \le 8$. Then employing Lemmas \ref{lem:property-of-f} and \ref{lem:f-divides-L(theta(f))}, one can use the following algorithm to compute the (conditionally) complete list of discriminants $D$ such that $\Cl{D}$ has exponent $E$.

\begin{itemize}
    \item First, compute the fundamental discriminants $D_0$ with class group of exponent dividing $E$. (It suffices to check these $D_0$, since $\Cl{D_0 f^2}$ surjects onto $\Cl{D_0}$ for all conductors $f$.) 
    \begin{itemize}
        \item[$\circ$] To know how high to compute, one can use  \cite[Theorem 1]{fieldexponents}, which (conditionally) gives the largest fundamental discriminant $D_0$ with class group of exponent dividing $E$.
    \end{itemize}
    \item Then for each fundamental discriminant $D_0$, compute the complete list of $f$ candidates: those satisfying the divisibility property $L(f) \mid (12 \cdot |\Or{D_0}^\times| \cdot E)$ from Lemma \ref{lem:property-of-f}.
    \begin{itemize}
        \item[$\circ$] By Lemma \ref{lem:f-divides-L(theta(f))}, $f \mid \theta(L(f))$. So to determine the complete list of $f$ candidates, one only needs to check the divisors of $\theta(m)$ for $m \mid (12 \cdot |\Or{D_0}^\times| \cdot E)$.
    \end{itemize}
    \item For each $f$ candidate, check if the exponent of $\Cl{D_0 f^2}$ is $E$.
    \begin{itemize}
        \item[$\circ$] To check very large $f$ candidates, it is most efficient to first verify that the exponent of $\Cl{D_0f'^2}$ divides $E$ for small divisors $f' \mid f$ (since for $f' \mid f$, $\Cl{D_0 f^2}$ surjects onto $\Cl{D_0 f'^2}$ by \cite[Corollary 7.17]{cox_primes_1989}). This optimization quickly excludes the vast majority of $f$ candidates.
    \end{itemize}
\end{itemize}

For exponents $1 \le E \le 8$, we used this algorithm to compute Tables \ref{table:CMexp1} - \ref{table:CMexp8} (see \cite{ross-code} for the code). This completes the proof of Theorem \ref{thm:exp-tables-complete-under-certain-assumptions}.
Observe that to compute fundamental discriminants $D_0$ with $\Cl{D_0}$ of exponent dividing $E$, the above algorithm uses the classifications from \cite[Theorem 1]{fieldexponents}. The completeness of each of these classifications depends on the conditions listed in Theorem \ref{thm:exp-tables-complete-under-certain-assumptions} (see \cite[Theorem 2]{fieldexponents}). And in fact, this is only reason we need such conditions; the rest of the algorithm is based on the unconditional results about conductors developed in this paper.

\section*{Acknowledgments}
This research was supported by NSF grant DMS-2349174. Hui Xue is supported by Simons Foundation grant MPS-TSM-00007911.

\appendix
\section{Tables of Discriminants} \label{appendix:tables}

This appendix contains the tables of discriminants we computed with class groups of small exponent.
For sake of space, the complete list of discriminants is omitted from Tables \ref{table:CMexp4}, \ref{table:CMexp6}, and \ref{table:CMexp8}. The full data can be found at \cite[Tables 1-8]{ross-code}.

\begin{table}[H]
\caption{Discriminants with class group of exponent 1.}
\begin{center}
\textbf{This table contains $13$ discriminants, with largest value $-163$. \\ \vspace{2mm}} 
\begin{tabular}{ c | p{125mm} } 
    \label{table:CMexp1}
    Class Group & Discriminants \\ 
    \hline
    $\{\textnormal{e}\}$ & $-3$, $-4$, $-7$, $-8$, $-11$, $-12$, $-16$, $-19$, $-27$, $-28$, $-43$, $-67$, $-163$ \\
    \hline
\end{tabular}
\end{center}
\end{table}

\begin{table}[H]
\caption{Discriminants with class group of exponent 2.}
\begin{center}
\textbf{This table contains $88$ discriminants, with largest value $-7392$. \\ \vspace{2mm}}
\begin{tabular}{ c | p{125mm} } 
    \label{table:CMexp2}
    Class Group & Discriminants \\ 
    \hline
    $\ZZ/2\ZZ$ & $-15$, $-20$, $-24$, $-32$, $-35$, $-36$, $-40$, $-48$, $-51$, $-52$, $-60$, $-64$, $-72$, $-75$, $-88$, $-91$, $-99$, $-100$, $-112$, $-115$, $-123$, $-147$, $-148$, $-187$, $-232$, $-235$, $-267$, $-403$, $-427$ \\
    \hline
    $(\ZZ/2\ZZ)^{2}$ & $-84$, $-96$, $-120$, $-132$, $-160$, $-168$, $-180$, $-192$, $-195$, $-228$, $-240$, $-280$, $-288$, $-312$, $-315$, $-340$, $-352$, $-372$, $-408$, $-435$, $-448$, $-483$, $-520$, $-532$, $-555$, $-595$, $-627$, $-708$, $-715$, $-760$, $-795$, $-928$, $-1012$, $-1435$ \\
    \hline
    $(\ZZ/2\ZZ)^{3}$ & $-420$, $-480$, $-660$, $-672$, $-840$, $-960$, $-1092$, $-1120$, $-1155$, $-1248$, $-1320$, $-1380$, $-1428$, $-1540$, $-1632$, $-1848$, $-1995$, $-2080$, $-3003$, $-3040$, $-3315$ \\
    \hline
    $(\ZZ/2\ZZ)^{4}$ & $-3360$, $-5280$, $-5460$, $-7392$ \\
    \hline
\end{tabular}
\end{center}
\end{table}

\begin{table}[H]
\caption{Discriminants with class group of exponent 3.}
\begin{center}
\textbf{This table contains $29$ discriminants, with largest value $-4027$. \\ \vspace{2mm}}
\begin{tabular}{ c | p{125mm} } 
    \label{table:CMexp3}
    Class Group & Discriminants \\ 
    \hline
    $\ZZ/3\ZZ$ & $-23$, $-31$, $-44$, $-59$, $-76$, $-83$, $-92$, $-107$, $-108$, $-124$, $-139$, $-172$, $-211$, $-243$, $-268$, $-283$, $-307$, $-331$, $-379$, $-499$, $-547$, $-643$, $-652$, $-883$, $-907$ \\
    \hline
    $(\ZZ/3\ZZ)^{2}$ & $-972$, $-1228$, $-2188$, $-4027$ \\
    \hline
\end{tabular}
\end{center}
\end{table}

\begin{table}[H]
\caption{Discriminants with class group of exponent 4.}
\begin{center}
\textbf{This table contains $485$ discriminants, with largest value $-887040$. \\ \vspace{2mm}}
\begin{tabular}{  c }
    \label{table:CMexp4}
    See \cite[Table 4]{ross-code}. 
\end{tabular}
\end{center}
\end{table}

\begin{table}[H]
\caption{Discriminants with class group of exponent 5.}
\begin{center}
\textbf{This table contains $31$ discriminants, with largest value $-37363$. \\ \vspace{2mm}}
\begin{tabular}{ c | p{125mm} } 
    \label{table:CMexp5}
    Class Group & Discriminants \\ 
    \hline
    $\ZZ/5\ZZ$ & $-47$, $-79$, $-103$, $-127$, $-131$, $-179$, $-188$, $-227$, $-316$, $-347$, $-412$, $-443$, $-508$, $-523$, $-571$, $-619$, $-683$, $-691$, $-739$, $-787$, $-947$, $-1051$, $-1123$, $-1723$, $-1747$, $-1867$, $-2203$, $-2347$, $-2683$ \\
    \hline
    $(\ZZ/5\ZZ)^{2}$ & $-12451$, $-37363$ \\
    \hline
\end{tabular}
\end{center}
\end{table}

\begin{table}[H]
\caption{Discriminants with class group of exponent 6.}
\begin{center}
\textbf{This table contains $1236$ discriminants, with largest value $-43522752$. \\ \vspace{2mm}}
\begin{tabular}{  c }
    \label{table:CMexp6}
    See \cite[Table 6]{ross-code}. 
\end{tabular}
\end{center}
\end{table}

\begin{table}[H]
\caption{Discriminants with class group of exponent 7.}
\begin{center}
\textbf{This table contains $40$ discriminants, with largest value $-118843$. \\ \vspace{2mm}}
\begin{tabular}{ c | p{125mm} } 
    \label{table:CMexp7}
    Class Group & Discriminants \\ 
    \hline
    $\ZZ/7\ZZ$ & $-71$, $-151$, $-223$, $-251$, $-284$, $-343$, $-463$, $-467$, $-487$, $-587$, $-604$, $-811$, $-827$, $-859$, $-892$, $-1163$, $-1171$, $-1372$, $-1483$, $-1523$, $-1627$, $-1787$, $-1852$, $-1948$, $-1987$, $-2011$, $-2083$, $-2179$, $-2251$, $-2467$, $-2707$, $-3019$, $-3067$, $-3187$, $-3907$, $-4603$, $-5107$, $-5923$ \\
    \hline
    $(\ZZ/7\ZZ)^{2}$ & $-63499$, $-118843$ \\
    \hline
\end{tabular}
\end{center}
\end{table}

\begin{table}[H]
\caption{Discriminants with class group of exponent 8.}
\begin{center}
\textbf{This table contains $2329$ discriminants, with largest value $-1723802080$. \\ \vspace{2mm}}
\begin{tabular}{  c }
    \label{table:CMexp8}
    See \cite[Table 8]{ross-code}. 
\end{tabular}
\end{center}
\end{table}


\bibliographystyle{plain}
\bibliography{biblio2}

\end{document}